\documentclass[12pt,twoside]{amsart}
\usepackage{amsmath, amsthm, amscd, amsfonts, amssymb, graphicx, color}
\usepackage[colorlinks = true, linkcolor = blue, urlcolor  = blue, citecolor = blue]{hyperref}
\usepackage[dvipsnames]{xcolor}
\usepackage{wrapfig}
\usepackage{subfig}
\def\ifif {if and only if\ \ }

\theoremstyle{definition}\newtheorem{thm}{Theorem}[section]
\theoremstyle{definition}\newtheorem{cor}[thm]{Corollary}
\theoremstyle{definition}\newtheorem{lem}[thm]{Lemma}
\theoremstyle{definition}\newtheorem{prop}[thm]{Proposition}
\theoremstyle{definition}
\theoremstyle{definition}
\theoremstyle{remark}
\theoremstyle{definition}\newtheorem{exam}[thm]{Example}
\theoremstyle{definition}

\begin{document}
\setcounter{page}{1}
%------------------------------------------------------------------------------------%
%%Don not change any thing in this part
%\hskip -1.4 cm
\vspace{-0.5cm}
\begin{center}
%\hskip 2 cm
\begin{tabular}{c}
\vspace{-0.4cm}
%\includegraphics[width=2cm]{images/cham}
 %{\scriptsize Faculty of Mathematical}\\
% {\scriptsize Sciences and Computer}\\
% 
\end{tabular}
\hskip 1.7cm
\begin{tabular}{l}
%\hline
\vspace{-0.2cm}
\scriptsize \rm\bf \textcolor{RoyalPurple}{Journal of Advanced Mathematical Modeling}\\
\vspace{-0.2cm}
\scriptsize \rm ISSN (print): 2251-8088, ISSN (on-line): 2645-6141\\
\vspace{-0.2cm}
\scriptsize Vol. {\bf\rm x} No. x {\rm(}2026{\rm)}, pp. xx-xx.\\
\vspace{0.2cm}
\scriptsize DOI: {\bf\rm x} No. 
%\hline
\end{tabular}
\hskip 1.7cm
\begin{tabular}{lcr}
\vspace{-0.4cm}
\end{tabular}
\end{center}
\noindent\rule{16cm}{0.6mm}
\vspace{0.7 cm}
%%%%%%%%%%%%%%%%%%%%%%%%%%%%%%%%%%%%%%%%%%%%%%%%%%%%%%%%%%%%%%%%%%%%%%%

\title[On $z^{\circ}$-ideals]{ On $z^{\circ}$-ideals and annihilator ideals}
\author{Ali Taherifar$^*$}
%\dedicatory{Dedicated to Professor . . . . . on the occasion of . . . . }
\thanks{{\scriptsize
\hskip -0.4 true cm MSC(2020): 16D25, 16D70, 06D22.
Received: dd mmmm yyyy, Accepted: dd mmmm yyyy.\\
$^*$Corresponding author}}
\maketitle

%------------------------------------------------------------------------------------%
%This part will be filled in by IJGT
\begin{center}
Communicated by:\;
\end{center}
%------------
%\begin{large}

\begin{abstract}
For $a\in R$, let $P_a$ denote  the intersection of all minimal prime ideals of $R$ containing $a$. An ideal $I$ of a ring $R$ is called a $z^{\circ}$-ideal if $P_a\subseteq I$ for all $a\in I$.
In this paper, we first investigate the class of $z^{\circ}$-ideals in non-commutative rings.  We provide characterizations of   $z^{\circ}$-ideals in 2-by-2 generalized triangular matrix rings, full and upper triangular matrix rings, and semiprime rings. Next, we explore new properties of the  lattice $rAnn(id(R))$, the set of
right annihilator ideals of $R$. We prove that $rAnn(id(R))$ forms an algebraic frame when $R$ is  semiprime and a coherent frame when $R$ is a reduced ring.  Furthermore, we characterize the smallest (resp., largest) right annihilator ideal containing (resp., contained
in) an ideal $I$ of an $SA$-ring $R$.
\end{abstract}
%\maketitle
\vskip 0.2 true cm
\section{Introduction}
Throughout this paper, $R$ denotes a nonzero associative ring with identity. The concept of $z^{\circ}$-ideals was first introduced in commutative rings. An ideal of a ring $R$ is called a $z^{\circ}$-ideal if $P_a\subseteq I$, for every $a\in I$, where $P_a$ is the
intersection of all minimal prime ideals containing $a$. The study of $z^{\circ}$-ideals has been pursued by many authors under different names. They were first investigated in \cite{S} in the context of Baer rings under the name Baer ideals. In \cite{BE}, the author referred to them as $z$-ideals. In \cite{BO}, they were studied under the name pseudo-normal ideals, while the author of \cite{CO} called them $B$-ideals.  In \cite{AMM},
they were referred to as $\zeta$-ideals. Additionally, these ideals have been examined in Riesz spaces, $f$-rings, and frames under the name $d$-ideals, as seen in \cite{BKW} (Section 3.3), \cite{LU} (Section 1), \cite{D}, \cite{HU}, \cite{L}, and \cite{DI}. Due to their similarity to $z$-ideals in $C(X)$ (see \cite{GJ}), the authors in \cite{AT, A, AAP, AZ, AKR, AKN, AK, AM2, TA} referred to them as $z^{\circ}$-ideals.

In this paper, we extend the definition of $z^{\circ}$-ideals from commutative rings to arbitrary rings. We recall in Section 2 the necessary background, and we fix notations. In Section 3, we investigate some properties of $z^{\circ}$-ideals in semiprime rings. Notably, we observe that understanding $z^{\circ}$-ideals in a ring $R$ can be reduced to studying them in semiprime rings.
We demonstrate that the class of $z^{\circ}$-ideals includes the class of $d$-ideals. Furthermore, we provide examples of $z^{\circ}$-ideals that are not $d$-ideals.

In Section 4, we  characterize  $z^{\circ}$-ideals in various ring extensions, including 2-by-2 generalized triangular matrix rings, as well as full and upper triangular matrix rings.

Section 5 focuses on the lattice of right annihilator ideals. In Theorem 5.3, we prove that when $R$ is a semiprime ring, the lattice of right annihilator ideals of $R$ ($rAnn(id(R))$) is a frame. For a reduced ring $R$, it is shown that $rAnn(id(R))$ is a coherent frame (Theorem 5.6). Additionally, in Theorem 5.10, we provide a characterization of the largest (resp., smallest) annihilator ideal containing (resp., contained
in)  an ideal $I$ of a right
$SA$-ring $R$. Several corollaries and examples illustrating this result are also presented.
\section{Background and Notation}
\subsection{Rings}
For any subset $S$ of $R$, l(S) and r(S) denote the left annihilator
and the right annihilator of a subset $S$ in $R$, respectively. An idempotent $e$ of $R$ is a left(right) semicentral
idempotent if $Re = eRe (eR = eRe)$. The ring of n-by-n (upper triangular)
matrices over $R$ is denoted by ($\bf T_n(R)) \bf M_n(R)$. A ring $R$ is called a \textit{right Ikeda–
Nakayama} (for short, a right $IN$-ring) if the left annihilator of the intersection of
any two right ideals is the sum of the left annihilators; that is, if $l(I)+ l(J)= l(I\cap J)$,
for all  right  ideals $I, J$ of $R$, see \cite{CN}. Recall from \cite{B} that a ring $R$ is called \textit{ right $SA$} if for each two  ideals $I, J$ of $R$, there exists an ideal $K$ of $R$ such that  $r(I) + r(J)= r(K)$, (see also \cite{HP}).
\subsection{Algebraic frame}
Our reference for frames  is \cite{PP}. A frame is a complete lattice $L$ satisfying the distributively law \[(\bigvee A)\wedge b=\bigvee \{a\wedge b: a\in A\},\] for any subset$A\subseteq L$ and $b\in L$. 

An element $a\in L$ is $compact$, if for any $X\subseteq L$, $a\leq\bigvee X$ implies that there is a finite subset $F$ of $X$ such that $a\leq\bigvee F$. We denote by  $\mathfrak{k}(L)$, the set of all compact elements of $L$. If every element of $L$ is the join of compact elements below it, then $L$ is said to be algebraic. If $a\wedge b\in \mathfrak{k}(L)$ for every $a, b\in \mathfrak{k}(L)$, then $L$ is said to have the \textit{finite intersection property}, throughout abbreviated as FIP.  If the top element of $L$ (which we shall denote by
1) is compact and $L$ has FIP, then $L$ is called coherent.
\section{Preliminary results and examples of $z^{\circ}$-ideals}

For any subset $A$ of $R$, let $P_A$ denote the intersection of all minimal prime ideals of $R$ that contain $A$. When $A=\{a\}$ consists of a single element, we write $P_a$ instead  of  $P_A$. Similar to the commutative case, an ideal $I$ of a ring $R$ is called a $z^{\circ}$-ideal if $P_a\subseteq I$ for every $a\in I$. 
By definition, every minimal prime ideal of $R$ is a $z^{\circ}$-ideal and the intersection of any collection of $z^{\circ}$-ideals is a $z^{\circ}$-ideal. Consequently, the prime radical ideal of 
$R$, denoted as $P(R)$ is the smallest $z^{\circ}$-ideal of $R$.  This implies that the structure of $z^{\circ}$-ideals of $R$ is equivalent to that of $R/P(R)$.  Thus, for the study of $z^{\circ}$-ideals of a ring $R$, we may assume that $P(R)=(0)$, meaning that $R$ can be considered a semiprime ring. We begin by recalling some well-known results from commutative rings and extend them to the non-commutative setting.
\begin{exam}
(1) If $I$ is a non-zero ideal (left ideal) in a semiprime ring $R$, then $l(I)$ is a $z^{\circ}$-ideal. Since, we observe that $l(I)=\bigcap_{P\in\bf Min(R), P\not\supseteq I}P$, by \cite[Lemma 11.40]{L1}. This implies  $Re$ (resp., $eR$) is a $z^{\circ}$-ideal, when $e$ is a right (resp., left) semicentral idempotent. Since, $er=ere$, for each $r\in R$ and hence $er(1-e)=0$. This shows $eR(1-e)=0$. Hence $Re=l((1-e)R)=l(R(1-e))$.

(2) Every ideal  in an strongly regular ring $R$ (i.e., a ring $R$ in which for each $a\in R$, there exists $b\in R$ such that $a=a^2b$)  is an intersection of 
minimal prime ideals and hence, it is a $z^{\circ}$-ideal. Since, every ideal in  an strongly regular ring $R$ is an intersection of minimal prime ideals.

(3) Let $R$ be a reduced ring. Then for every minimal left ideal $I$ of $R$, there is an idempotent $e\in R$ such that $I=Re=l((1-e)R)=l(R(1-e))$, by  \cite[Lemma 10.22]{L1}, and hence $IR$ is a $z^{\circ}$-ideal, by Part (1). 
\end{exam}
%%%%%%%%%%%%%%%%%%%%%%%%%%%%%%%%%%%%%%%%%%%%%%%%%%%%%%%
Mason in \cite{M} defined an ideal $I$ of a reduced ring $R$ as a $d$-ideal if for each $a\in I$, $r(l(a))\subseteq I$. We extend this definition to arbitrary rings. Specifically, an ideal $I$ of a ring $R$ is called a right  $d$-ideal if $r(l(RaR))\subseteq I$ for each $a\in I$. Since in a reduced ring $R$, we have $rl(RaR)=r(l(a))$, this provides a natural extension of Mason’s definition.  We observe that the class of $d$-ideals and $z^{\circ}$-ideals in a non-semiprime ring can be distinct.  Consider the ring $\Bbb Z_{12}$  (ring of integers modulo 12). Then, it is easy to see that the ideal $<\overline{4}>$ is a $d$-ideal but not a $z^{\circ}$-ideal. Since, $P_{\overline{4}}=<\overline{2}>\not\subseteq <\overline{4}>$.
%%%%%%%%%%%%%%%%%%%%%%%%%%%%%%%%%%
\begin{prop}\label{ha}
For a ring $R$ the following statements are equivalent.
\begin{enumerate}
\item For each $a\in R$, $P_a\subseteq rl(RaR)$.

\item The ring $R$ is semiprime.

\item For $a, b\in R$, $P_b\subseteq P_a$ implies $l(RaR)\subseteq l(RbR)$.

\item Every right $d$-ideal is a $z^{\circ}$-ideal.

\item Every right annihilator ideal of $R$ is a $z^{\circ}$-ideal.
\end{enumerate}
\end{prop}
\begin{proof}

(1)$\Rightarrow$(2) We have $P_{\circ}\subseteq r(l(0))=0$. This says the intersection of all minimal prime ideals of $R$ is zero, i.e., $R$ is a semiprime ring.

(2)$\Rightarrow$(3) Let $P_b\subseteq P_a$ and $x\in l(RaR)$. Then $xRaR=0$ and hence $P_{xRbR}\subseteq P_{xRaR}=P_0=0$. This implies $xRbR=0$. Thus $x\in l(RbR)$. Therefore, $l(RaR)\subseteq l(RbR)$.

(3)$\Rightarrow$(4) Let $I$ be a right $d$-ideal, $a\in I$ and $b\in P_a$. Then $P_b\subseteq P_a$. By hypothesis, $l(RaR)\subseteq l(RbR)$. This implies $b\in r(l(RbR))\subseteq rl(RaR)\subseteq I$.

(4)$\Rightarrow$(5) Suppose that $I$ is a right annihilator and $a\in I$. Then $r(l(RaR))\subseteq I$. It is easy to see that $r(l(RaR))$ is a right $d$-ideal, and so by hypothesis, is  a $z^{\circ}$-ideal. Thus, $a\in rl(RaR)$ implies $P_a\subseteq rl(RaR)\subseteq I$. That is  $I$ is a $z^{\circ}$-ideal.

(5)$\Rightarrow$(1)  For each $a\in R$, $r(l(RaR))$ is a $z^{\circ}$-ideal, by hypothesis. Thus,  $a\in rl(RaR)$ implies that $P_a\subseteq rl(RaR)$.  
\end{proof}
%%%%%%%%%%%%%%%%%%%%%%%%%%%%%%%%%%%%%%%%%
\begin{cor}\label{sara}
The class of $z^{\circ}$-ideals and right $d$-ideals in a ring $R$ coincide \ifif for every $a\in R$, $r(l(RaR))=P_a$.
\end{cor}
\begin{proof}
$\Rightarrow$ By proposition \ref{ha}, for each $a\in R$, we have $P_a\subseteq rl(RaR)$. Now, by hypothesis, $P_a$ is a right $d$-ideal and $a\in P_a$, it follows that $r(l(RaR))\subseteq P_a$, so we are done.

$\Leftarrow$ Again, by Proposition \ref{ha}, every right $d$-ideal is a $z^{\circ}$-ideal. Now, assume $I$ is a $z^{\circ}$-ideal and $a\in I$. Then, by definition, $P_a\subseteq I$. By hypothesis, $r(l(RaR))=P_a$, which implies $r(l(RaR))\subseteq I$, i.e., $I$ is a right $d$-ideal. This completes the proof.
\end{proof}
%%%%%%%%%%%%%%%%%%%%%%%%%%%%%%%%%%%%%%%%%%%%%%%%%%%%%%%
\begin{cor}\label{12}
Let $R$ be a reduced ring.
\begin{enumerate}
\item For each $a\in R$, $P_a=r(l(RaR))=r(l(a))$.

\item The class of $z^{\circ}$-ideals and right $d$-ideals coincide. 

\item A principal ideal $RaR$ is a right annihilator ideal \ifif it is a $z^{\circ}$-ideal.
\end{enumerate}
\end{cor}
\begin{proof}
(1) The second equality is evident. By Proposition \ref{ha}, we have $P_a\subseteq rl(RaR)$. Now, suppose $x\in rl(RaR)$ and let $P\in\bf Min(R)$ with $a\in P$. Then $l(RaR)x=0$ (i.e., $l(RaR)RxR=0$) and by \cite[Lemma 3.1]{KLN}, we have $l(RaR)=l(a)\not\subseteq P$. These conditions imply that $x\in P$, and hence $x\in P_a$.

(2) Follows from Part (1) and Corollary \ref{sara}.

(3) If $RaR$ is a right annihilator ideal, then $r(l(RaR))=RaR$. By Part (1), we have $P_a=r(l(RaR))$, which implies $P_a=RaR$, meaning that $RaR$ is a $z^{\circ}$-ideal. Conversely, suppose $RaR$ is  a $z^{\circ}$-ideal. Since $a\in RaR$, it follows that $P_a\subseteq RaR$. Again, by Part (1), we have $P_a=r(l(RaR))$, and thus $RaR=r(l(RaR))$.
\end{proof}
%%%%%%%%%%%%%%%%%%%%%%%%%%%%%%%%%%%%%%%%%%%%%%%%%%%%
\begin{lem}\label{into}
Let $I$ be an ideal (right ideal) of a semiprime ring $R$, $a\in R$  and $l(I)\subseteq l(a)$. Then, for each $y\in R$, $l(Iy)\subseteq  l(ay)$.
\end{lem}
\begin{proof}
Let $x\in l(Iy)$. Then $xIy=0$. Thus $RxIy=0$. This implies $(RyRxI)^2=RyRxIRyRxI=RyRxIyRxI=0$.  By semiprime hypothesis, $RyRxI=0$. This says $RyRx\subseteq l(I)\subseteq l(a)$. Hence, $RyRxa=0$. This implies $(RxaRy)^2=0$, and hence by hypothesis, $RxaRy=0$, i.e., $xay=0$. Therefore, $x\in l(ay)$.
\end{proof} 
%%%%%%%%%%%%%%%%%%%%%%%%%%%%%%%%%%%%%%%%%
\begin{prop}
Let $I$ and $P$ be ideals of a reduced ring $R$ and $P$ be a prime ideal. If $I\cap P$ is a $z^{\circ}$-ideals, then either $I$ or $P$ is a $z^{\circ}$-ideal. In particular, if $P$ and $Q$ are prime ideals of $R$ which are not in a chain and $P\cap Q$ is a $z^{\circ}$-ideal, then both $P$ and $Q$ are $z^{\circ}$-ideals.
\end{prop}
\begin{proof}
If $I\subseteq P$, then $I\cap P=I$ is a $z^{\circ}$-ideal. Suppose $I\not\subseteq P$, $a\in P$ and $l(RaR)\subseteq l(RxR)$. Take $y\in I\setminus P$. By Lemma \ref{into}, $l(RaRy)\subseteq l(RxRy)$. Since $RaRy\subseteq I\cap P$ and $I\cap P$ is a $z^{\circ}$-ideal, it follows from Corollary \ref{12} that $RxRy\subseteq I\cap P$. Consequently, $RxRy\subseteq P$, which implies 
$x\in P$. The proof of the second part of the proposition follows similarly. 
\end{proof}
%%%%%%%%%%%%%%%%%%%%%%%%%%%%%%%%%%%%%%%%%%%%%%%%%%%%%%%%%%%%%%%
%%%%%%%%%%%%%%%%%%%%%%%%%%%%%%%%%%%%%%%%%%%%%%%%%%
\begin{exam}
Let $R=\begin{pmatrix}
  \Bbb Z_{6} & \Bbb Z_{6} \\
  0 & \Bbb Z_{6} \\
\end{pmatrix}$. Then, by Theorem \ref{3.6},  $I
=\begin{pmatrix}
  <\overline{3}> & \Bbb Z_{6} \\
  0 & <\overline{2}>\\
\end{pmatrix}$  is a $z^{\circ}$-ideal in $R$, where $<\overline{2}>$ and $<\overline{3}>$ are ideals generated by $\overline{2}$ and $\overline{3}$, respectively. However,  the ideal $I$ is not a right $d$-ideal. For, $A=\begin{pmatrix}
  \overline{3} & \overline{1} \\
  0 & 0\\
\end{pmatrix}\in I$ and we have $RA=\begin{pmatrix}
  <\overline{3}> & \Bbb Z_{6} \\
  0 & 0\\
\end{pmatrix}$. Thus $l(RAR)=l(RA)=\begin{pmatrix}
0 & \Bbb Z_{6}\\
  0 & \Bbb Z_{6}\\
\end{pmatrix}$ and this implies that $r(l(RAR))=\begin{pmatrix}
  \Bbb Z_{6} & \Bbb Z_{6} \\
  0 & 0 \\
\end{pmatrix}\not\subseteq I$.
\end{exam}

\section{$z^{\circ}$-ideals in extension rings}
In this section, we determine $z^{\circ}$-ideals in 2-by-2
generalized triangular matrix rings, full and upper triangular matrix rings. First, we need to extend the definition of $z^{\circ}$-ideals as the following.

For $n\in\Bbb N$, we call an ideal $I$ of $R$ a $z_{n}^{\circ}$-ideal if for every finite subset $F$ of $I$ with at most $n$ elements, we have $P_F\subseteq I$.
Clearly, $z_{1}^{\circ}$-ideals precisely are the $z^{\circ}$-ideals. Moreover, an ideal $I$ of $R$ is a $sz^{\circ}$-ideal (see \cite{AM1}) \ifif $I$ is a $z_{n}^{\circ}$-ideal for each $n\in\Bbb N$.  For $n>1$, every $z_{n}^{\circ}$-ideal is also a $z_{n-1}^{\circ}$-ideal, and hence a $z^{\circ}$-ideal. However, the converse does not necessarily hold. Indeed, there exists an example of a $z^{\circ}$-ideal (hence $z_{1}^{\circ}$-ideal) that is not an $sz^{\circ}$-ideal and consequently, not a $z_{n}^{\circ}$-ideal for some $n>1$, see \cite[Example 4.2]{AM1}.
\begin{lem}\label{1}
Let $A=[a_{ij}]\in \bf M_{n}(R)$.\\ Then $P_A=\bf M_{n}(P_B)$, where $B=\{a_{ij}: 1\leq i, j\leq n\}.$
\end{lem}
\begin{proof}
By \cite[Theorem 3.1]{L}, every ideal of $\bf M_{n}(R)$ is of the form $M_n(I)$ for some ideal $I$ of $R$. Moreover, $M_n(I)$ is a minimal prime ideal of $\bf M_n(R)$ \ifif $I$ is a minimal prime ideal of $R$. By definition, we have: \[P_A=\bigcap_{A\in \bf M_n(P)\in\bf Min(\bf M_n(R))}\bf M_n(P)=\bf M_n(\bigcap_{A\in \bf M_n(P)\in\bf Min(\bf M_n(R))}P)=\]\[\bf M_n(\bigcap_{a_{ij}\in P\in\bf Min(R), 1\leq i, j\leq n}P)=\bf M_n(P_{\{a_{ij}: 1\leq i, j\leq n\}})=\bf M_n(P_B).\]
\end{proof}
\begin{thm}\label{111}
An ideal $J$ of $\bf M_{n}(R)$ is a $z^{\circ}$-ideal \ifif $J= M_n(I)$ for some $z_{n^2}^{\circ}$-ideal $I$ of $R$.
\end{thm}
\begin{proof}
$\Rightarrow$ Let $J$ be a $z^{\circ}$-ideal of $\bf M_{n}(R)$. By \cite[Theorem 3.1]{L}, $J= M_n(I)$ for some ideal $I$ of $R$. We claim that $I$ is a $z_{n^2}^{\circ}$-ideal of $R$. Let $F$ be a subset of $I$ with at most $n^2$ elements. Without loss of generality, we assume   $F=\{a_{11}, a_{12},....a_{1n}, a_{21}, a_{22}, ..., a_{2n}, ...., a_{n1}, a_{n2}, ..., a_{nn}\}$. Put $A= [a_{ij}]$, so that $A\in \bf M_{n}(R)$. By Lemma \ref{1}, we obtain 
$P_A=\bf M_n(P_F)$.  By hypothesis, $P_A\subseteq J$, which implies $P_F\subseteq I$. Hence, $I$ is a $z_{n^2}^{\circ}$-ideal of $R$.

$\Leftarrow$ Let $I$ be a $z_{n^2}^{\circ}$-ideal of $R$ and let $A=[a_{ij}]\in J=M_n(I)$. Define $F=\{a_{ij}: 1\leq i, j\leq n\}$.  By Lemma \ref{1},  $P_A= \bf M_n(P_F)$. Since $F\subseteq I$ and contains at most $n^2$ elements and $I$ is a $z_{n^2}^{\circ}$-ideal, we have $P_F\subseteq I$, which implies  $P_A=\mathbf{{M_n(P_F)}}\subseteq  M_n(I)=J$.
\end{proof}
\begin{exam}
Consider the ring $R=\Bbb Z_{8}$. Then, it is easy to see that $I=<\overline{2}>$ is a $z^{\circ}_{4}$-ideal of $R$. Now, by Theorem \ref{111}, $M_{2}(<\overline{2}>)$ is a $z^{\circ}$-ideal of $M_{2}(\Bbb Z_{8})$.
\end{exam}
By \cite[Theorem 3.2]{taheri}, every ideal of $\bf T_n(R)$ is of the following form: \[I=\begin{pmatrix}
  J_{11} & J_{12}&  J_{13}&. & . & . & J_{1n} \\
  0 & J_{22}& J_{23} & . & . &.& J_{2n} \\
  . & . & . & . & . & .&. \\
  . & . & . & . & . & .&. \\
  . & . & . & . & . & .&. \\
  0 & 0 & . & . & .&0 & J_{nn} \\
\end{pmatrix}, J_{ik}\subseteq J_{ik+1}\quad\text{and}\ J_{i+1 k}\subseteq
J_{ik}.\] It is easy to see that $I$ is a prime ideal of $\bf T_n(R)$ \ifif all $J_{ij}=R$ ($1\leq i, j\leq n$) except one of $J_{ii}$ which is a prime ideal of $R$.
\begin{lem}\label{200}
Let $A=[a_{ij}]\in \bf T_n(R)$. Then \[P_A=\begin{pmatrix}
  P_{a_{11}} & R&  R&. & . & . & R \\
  0 & P_{a_{22}}& R & . & . &.& R \\
  . & . & . & . & . & .&. \\
  . & . & . & . & . & .&. \\
  . & . & . & . & . & .&. \\
  0 & 0 & . & . & .& 0 & P_{a_{nn}} \\
\end{pmatrix}\]
\end{lem}
\begin{proof}
 Let $A=[a_{ij}]\in \bf T_n(R)$. Then  \[P_A=(\bigcap_{a_{11}\in P\in\bf Min(R)} \begin{pmatrix}P & R&  R&. & . & . & R \\
  0 & R& R & . & . &.& R \\
  . & . & . & . & . & .&. \\
  . & . & . & . & . & .&. \\
  . & . & . & . & . & .&. \\
  0 & 0 & . & . & .& 0 & R \\
\end{pmatrix})\bigcap(\bigcap_{a_{22}\in P\in\bf Min(R)} \begin{pmatrix}R & R&  R&. & . & . & R \\
  0 & P& R & . & . &.& R \\
  . & . & . & . & . & .&. \\
  . & . & . & . & . & .&. \\
  . & . & . & . & . & .&. \\
  0 & 0 & . & . & .& 0 & R \\
\end{pmatrix})...\]\[
\bigcap(\bigcap_{a_{nn}\in P\in\bf Min(R)}\begin{pmatrix}R & R&  R&. & . & . & R \\
  0 & R& R & . & . &.& R \\
  . & . & . & . & . & .&. \\
  . & . & . & . & . & .&. \\
  . & . & . & . & . & .&. \\
  0 & 0 & . & . & .& 0 & P \\
\end{pmatrix})=\begin{pmatrix}\bigcap_{a_{11}\in P\in\bf Min(R)}P & R&  R&. & . & . & R\\
  0 & R& R & . & . &.& R \\
  . & . & . & . & . & .&. \\
  . & . & . & . & . & .&. \\
  . & . & . & . & . & .&. \\
  0 & 0 & . & . & .& 0 & R \\
\end{pmatrix}\]\[\bigcap \begin{pmatrix}R & R&  R&. & . & . & R \\
  0 & \bigcap_{a_{22}\in P\in\bf Min(R)}P& R & . & . &.& R \\
  . & . & . & . & . & .&. \\
  . & . & . & . & . & .&. \\
  . & . & . & . & . & .&. \\
  0 & 0 & . & . & .& 0 & R \\
\end{pmatrix}...
\bigcap\begin{pmatrix}R & R&  R&. & . & . & R \\
  0 & R& R & . & . &.& R \\
  . & . & . & . & . & .&. \\
  . & . & . & . & . & .&. \\
  . & . & . & . & . & .&. \\
  0 & 0 & . & . & .& 0 & \bigcap_{a_{nn}\in P\in\bf Min(R)}P \\
\end{pmatrix}=\]\[\begin{pmatrix}P_{a_{11}} & R&  R&. & . & . & R \\
  0 & P_{a_{22}}& R & . & . &.& R \\
  . & . & . & . & . & .&. \\
  . & . & . & . & . & .&. \\
  . & . & . & . & . & .&. \\
  0 & 0 & . & . & .& 0 & P_{a_{nn}} \\
\end{pmatrix}.\]
\end{proof}
%%%%%%%%%%%%%%%%%%%%%%%%%%%%%%%%%%%%%%%%%%%%%%%%%%%%%%%%%%%%%%%%
\begin{thm}\label{2}
The ideal $I=\begin{pmatrix}
  J_{11} & J_{12}&  J_{13}&. & . & . & J_{1n} \\
  0 & J_{22}& J_{23} & . & . &.& J_{2n} \\
  . & . & . & . & . & .&. \\
  . & . & . & . & . & .&. \\
  . & . & . & . & . & .&. \\
  0 & 0 & . & . & .&0 & J_{nn} \\
\end{pmatrix}$ of $\bf T_n(R)$ is a $z^{\circ}$-ideal \ifif each $J_{ii}$ ($1\leq i\leq n$) is a $z^{\circ}$-ideal and $J_{ij}=R$ for all $j>i$.
\end{thm}
\begin{proof}
$\Rightarrow$ First, we note that $[0]\in I$. By Lemma \ref{200}, this implies $P_{[0]}=[P_{ij}]\subseteq I$, where $P_{ii}=P_0$ for all $1\leq i\leq n$ and $P_{ij}=R$ whenever $j> i$. Consequently, we have $J_{ij}=R$ for all $j>i$. Next, we claim that $J_{ii}$ (for $1\leq i\leq n$) is a $z^{\circ}$-ideal. Fix $1\leq i\leq n$ and let $a\in J_{ii}$. Define the matrix $A=[a_{ij}]$ as follows: \\ $a$ occupies the (i,i) position,
  $a_{jj}=0$ for all $1\leq j\neq i\leq n$, $a_{ij}=0$ for $i>j$, and all other entries are 1.\\ Since $A\in I$,  Lemma \ref{200} gives us, $P_A=[P_{ij}]$, where $P_{ii}=P_a$, $P_{jj}=P_0$ for all $1\leq j\neq i\leq n$, $P_{ij}=0$ for $i>j$ and elsewhere is  $R$ . By hypothesis, $P_A\subseteq I$, which  implies $P_a\subseteq J_{ii}$. Thus $J_{ii}$ is a $z^{\circ}$-ideal.

$\Leftarrow$ Let $A=[a_{ij}]\in I$. By hypothesis, \[I=\begin{pmatrix}
  J_{11} & R&  R&. & . & . & R \\
  0 & J_{22}& R & . & . &.& R \\
  . & . & . & . & . & .&. \\
  . & . & . & . & . & .&. \\
  . & . & . & . & . & .&R \\
  0 & 0 & . & . & .&0 & J_{nn}\\
\end{pmatrix}.\]By Lemma \ref{200}, \[P_A=\begin{pmatrix}P_{a_{11}} & R&  R&. & . & . & R \\
  0 & P_{a_{22}}& R & . & . &.& R \\
  . & . & . & . & . & .&. \\
  . & . & . & . & . & .&. \\
  . & . & . & . & . & .&. \\
  0 & 0 & . & . & .& 0 & P_{a_{nn}} \\
\end{pmatrix}.\] Since for each $1\leq i\leq n$, $J_{ii}$ is a $z^{\circ}$-ideal, $P_{a_{ii}}\subseteq J_{ii}$. Hence $P_A\subseteq I$.
\end{proof}
\begin{exam}
 (1) By Theorem \ref{2}, the zero-ideal of $T_n(R)$ is not a $z^{\circ}$-ideal, even when $R$ is a semiprime ring.\\
  (2) Let $R$ be a semiprime  ring. Then the zero ideal is a $z^{\circ}$-ideal in $R$. Hence, by theorem \ref{2}, the following ideal is a $z^{\circ}$-ideal in $\bf T_n(R)$.\[\begin{pmatrix}0 & R&  R&. & . & . & R \\
  0 & 0& R & . & . &.& R \\
  . & . & . & . & . & .&. \\
  . & . & . & . & . & .&. \\
  . & . & . & . & . & .&R \\
  0 & 0 & . & . & .& 0 & 0 \\
\end{pmatrix}.\]
\end{exam}
 We are including the following lemma for
completeness since it is used in the next result. 
\begin{lem}\label{fat}
An ideal $J=\begin{pmatrix}
  I & N \\
  0 & L \\
\end{pmatrix}$  of $T=\begin{pmatrix}
  S & M \\
  0 & R \\
\end{pmatrix}$  is a prime ideal \ifif

(i) $N=M$.\\\indent (ii) $I=S$ and $L$ is a prime ideal of $R$
or $L=R$ and $I$ is a prime ideal of $S$.
\end{lem}
\begin{proof}
$\Rightarrow$ By \cite[Proposition 1.17]{L}, $K=\begin{pmatrix}
  0 & M\\
  0 & 0 \\
\end{pmatrix}$ is an ideal of $T$. Since $K^2=0\subseteq J$, the hypothesis implies, $K\subseteq J$, which gives $N=M$. Again by \cite[Proposition 1.17]{L}, $I_1=\begin{pmatrix}
  I & M\\
  0 & R \\
\end{pmatrix}$ and $I_2=\begin{pmatrix}
  S & M\\
  0 & L \\
\end{pmatrix}$ are  ideals of $T$ and $I_1I_2=\begin{pmatrix}
  I & M\\
  0 & L \\
\end{pmatrix}$=J. Thus, we must have either  $I_1\subseteq J$ or $I_2\subseteq J$, which implies $I=S$ or $L=R$. Now, assume $I=S$. Let $L_1, L_2$ be two ideals of $R$ such that  $L_1L_2\subseteq L$. Then, $H_1=\begin{pmatrix}
  0 & M\\
  0 & L_1 \\
\end{pmatrix}$  and $H_2=\begin{pmatrix}
  0 & M\\
  0 & L_2 \\
\end{pmatrix}$ are two ideals of $T$  and  $H_1H_2\subseteq J$. Consequently, we must have either $H_1\subseteq J$ or $H_2\subseteq J$, which implies $L_1\subseteq L$ or $L_2\subseteq L$. Similarly, when $L=R$, we conclude  that $I$ is a prime ideal of $R$.

$\Leftarrow$ It is easy to see that whenever $I$ and $L$ are prime ideals of $S$ and $R$, respectively, the ideals $\begin{pmatrix}
  S & M\\
  0 & L\\
\end{pmatrix}$ and $\begin{pmatrix}
  I & M\\
  0 & R\\
\end{pmatrix}$ are prime ideals of $T$.
\end{proof}
\begin{lem}\label{3}
Let $A=\begin{pmatrix}
  a_{11} & a_{12}\\
  0 & a_{22} \\
\end{pmatrix}\in T$. Then $P_A=\begin{pmatrix}
  P_{a_{11}} & M \\
  0 & P_{a_{22}} \\
\end{pmatrix}$.
\end{lem}
\begin{proof}
By Lemma \ref{fat}, we have \[P_A=(\bigcap_{a_{11}\in P\in\bf Min(S)}\begin{pmatrix}
  P & M\\
  0 & R \\
\end{pmatrix})\bigcap(\bigcap_{a_{22}\in P\in\bf Min(R)}\begin{pmatrix}
  S & M\\
  0 & P \\
\end{pmatrix})=\]\[\begin{pmatrix}
  \bigcap_{a_{11}\in P\in\bf Min(S)}P & M\\
  0 & \bigcap_{a_{22}\in P\in\bf Min(R)}P\\
\end{pmatrix}=\begin{pmatrix}
  P_{a_{11}} & M\\
  0 & P_{a_{22}} \\
\end{pmatrix}.\]
\end{proof}
The following result provides a characterization of $z^{\circ}$-ideals in a 2-by-2 generalized triangular matrix ring.
%%%%%%%%%%%%%%%%%%%%%%%%%%%%%%%%%%%%%%%%%%%%%%%%%%%%%%%%%%%%
\begin{thm}\label{3.6}
Let $J=\begin{pmatrix}
  I & N \\
  0 & L \\
\end{pmatrix}$ be an ideal of $T=\begin{pmatrix}
  S & M \\
  0 & R \\
\end{pmatrix}$. \\\indent Then $J$ is a $z^{\circ}$-ideal of $T$ \ifif 

(i) $N=M$.\\\indent (ii) Two ideals $I, L$ are $z^{\circ}$-ideals of $S$ and $R$, respectively.
\end{thm}
\begin{proof}
$\Rightarrow$ Let $n\in\ N$. Consider the matrix $A=\begin{pmatrix}
  0_{S} & n \\
  0 & 0_{R} \\
\end{pmatrix}\in J$. By hypothesi and Lemma \ref{3}, we have $P_A=\begin{pmatrix}
  P_0 & M \\
  0 & P_0 \\
\end{pmatrix}\subseteq J$. This implies that $M=N$. Now, let $a\in I$ and $b\in L$. Consider the matrix $B=\begin{pmatrix}
  a & n \\
  0 & b \\
\end{pmatrix}\in J$ for each $n\in M$.  By assumption, we must have $P_B\subseteq J$, which implies that $P_a\subseteq I$ and $P_b\subseteq L$. This shows that  $I$ and $L$  are $z^{\circ}$-ideals of $S$ and $R$, respectively.

$\Leftarrow$ Let $A=\begin{pmatrix}
  a & m \\
  0 & b\\
\end{pmatrix}\in \begin{pmatrix}
  I & M \\
  0 & L \\
\end{pmatrix}$. Then $a\in I$ and $b\in L$. By hypothesis, $P_a\subseteq I$, $P_b\subseteq L$ and by Lemma \ref{3}, $P_A=\begin{pmatrix}
  P_a & M \\
  0 & P_b\\
\end{pmatrix}$. Thus $P_A\subseteq J$.
\end{proof}
%%%%%%%%%%%%%%%%%%%%%%%%%%%%%%%%%%%%%%%%%%%%%%%%%%%%%%%%%%%%%%%%%%%%%%%%%%%%%%%%%

\section{On the lattice of annihilator ideals}

An ideal $I$ of a ring $R$ is a right (left) annihilator ideal if  rl(I)=I (lr(I)=I);
equivalently, $l(I)\subseteq l(x)$ ($r(I)\subseteq r(x)$) and $x\in R$, imply $x\in I$.

Put $rAnn(id(R))=\{I: I$ is a  right annihilator ideal of $R\}$. In \cite{B} and \cite{DT} it is shown that  $rAnn(id(R))$ is a complete lattice with the following operations.  \[I\vee J=r_R(l_R(I)\cap l_R(J))\quad \text{ and }\quad I\wedge J=I\cap J.\]
In this section, we try to investigate some new properties of this lattice.\\
In the following result some items are well-known. However, we bring their proofs. 
\begin{lem}\label{20}
The following statements are equivalent.
\begin{enumerate}
\item The ring $R$ is semiprime.

\item For every two ideals $I, J$ of $R$, $l(IJ)=l(I\cap J)$ ($r(IJ)=r(I\cap J)$).

\item For every two ideals $I, J$ of $R$, $rl(IJ)=rl(I\cap J)$.
\end{enumerate}
\end{lem}
\begin{proof}
(1)$\Rightarrow$(2) Clearly $l(I\cap J)\subseteq l(IJ)$, since $IJ\subseteq I\cap J$.  For the other inclusion, consider $r\in l(IJ)$ and $s\in I\cap J$. Then $RsRsR\subseteq IJ$ and $RrR\subseteq l(IJ)$ and  hence we have  $(RrRsR)^2=RrRsRrRsR\subseteq RrRsRsR=0$. By assumption, this shows that $RrRsR=0$ and thus $rs=0$. Therefore,  $r\in l(I\cap J)$. So we are done.

(2)$\Rightarrow$(1) Let $I$ be an ideal of $R$ and $I^2=0$. Then we have  \[l(I)=l(I\cap I)=l(I^2)=R.\] This implies that  $I=0$.

(2)$\Rightarrow$(3) Trivial.

(3)$\Rightarrow$(1) Let $I$ be an ideal of $R$ and $I^2=0$. Then, by hypothesis,  $rl(I)=rl(I\cap I)=rl(I^2)=0$. Thus, $l(I)=lrl(I)=R$, i.e., $I=0$.
\end{proof}
%%%%%%%%%%%%%%%%%%%%%%%%%%%%%%%%%%%%%%%%%%%%%%%%%%%%%%%%%%%%%%%%%%
\begin{lem}\label{KH}
Let $R$ be a semiprime. Then,   for any family $\{I_{\alpha}: \alpha\in S\}$ of ideals of $R$ and any right annihilator ideal $J$ of $R$, we have \[J\cap rl(\sum_{\alpha\in S}I_{\alpha})=rl(\sum_{\alpha\in S}I_{\alpha}\cap J).\]
\end{lem}
\begin{proof}
  $\Rightarrow$ Since $J$ is a right annihilator ideal, we have \[rl(\sum_{\alpha\in S}I_{\alpha}\cap J)\subseteq J\cap rl(\sum_{\alpha\in S}I_{\alpha}).^{(1)}\] On the other hand, we have \[\sum_{\alpha\in S}(J\cap I_{\alpha})\subseteq rl(\sum_{\alpha\in S}J\cap I_{\alpha}).\] This implies that for each $\alpha\in S$, we have $J\cap I_{\alpha}\subseteq rl(\sum_{\alpha\in S}J\cap I_{\alpha})$. By hypothesis, this implies  \[J\cap I_{\alpha}\cap l(\sum_{\alpha\in S}J\cap I_{\alpha})=0, \quad\text{for each}\quad \alpha\in S,\] and hence $J\cap l(\sum_{\alpha\in S}J\cap I_{\alpha})\subseteq l(I_{\alpha})$, for each $\alpha\in S$. Taking the intersection over all $\alpha$ we obtain: \[J\cap l(\sum_{\alpha\in S}J\cap I_{\alpha})\subseteq \bigcap_{\alpha\in S}l(I_{\alpha})=l(\sum_{\alpha\in S}I_{\alpha}).\] This leads to $J\cap l(\sum_{\alpha\in S}J\cap I_{\alpha})\cap rl(\sum_{\alpha\in S}I_{\alpha})=0$. Thus, \[J\cap rl(\sum_{\alpha\in S}I_{\alpha})\subseteq rl(\sum_{\alpha\in S}J\cap I_{\alpha}).^{(2)}\] Combining (1) and (2), we establish the required equality.
\end{proof}
\begin{thm}\label{H}
Let $R$ be a semiprime. Then,   $<rAnn(id(R)), \subseteq>$ is a frame.
\end{thm}
\begin{proof}
Let $J$ and the family $\{I_{\alpha}: \alpha\in S\}$ be right annihilator ideals. Then, by Lemma \ref{KH},  \[J\wedge (\vee_{\alpha\in S}I_{\alpha})=J\cap rl(\sum_{\alpha\in S}I_{\alpha})=rl(\sum_{\alpha\in S}J\cap I_{\alpha})=\bigvee_{\alpha\in S}(J\wedge I_{\alpha}).\] This shows that $rAnn(id(R))$ is a frame.
\end{proof}
\begin{lem}\label{21}
Let $R$ be a semiprime ring.  Then the compact elements of $rAnn(id(R))$ are precisely the
ideals of the form $rl(Ra_1R)\vee...\vee rl(Ra_nR)$, for some finitely many elements $a_i\in R$. Indeed, \[\mathfrak{k}({rAnn(id(R))})=\{rl(Ra_1R)\vee...\vee rl(Ra_nR): a_1, ..., a_n\in R, n\in\Bbb N\}.\]
\end{lem}
\begin{proof}
To show that each $rl(Ra_1R)\vee...\vee rl(Ra_nR)$ is compact, we first show that each $rl(RaR)$ is compact. Let $\{I_{\alpha}: \alpha\in S\}$ be a  directed collection of right annihilator ideals with $rl(RaR)\leq\bigvee_{\alpha\in S}I_{\alpha}$. Then $rl(RaR)\subseteq\bigcup_{\alpha\in S}I_{\alpha}$. Since $a\in rl(RaR)$, we have $a\in I_{\alpha}$, for some $\alpha\in S$. This implies $rl(RaR)\subseteq I_{\alpha}$, since $I_{\alpha}$ is a right annihilator ideal. Therefore, $rl(RaR)$ is compact, and hence  $rl(Ra_1R)\vee...\vee rl(Ra_nR)$ is compact. Now, let $K\in \mathfrak{k}{(rAnn(id(R)))}$. Since $K$ is a right annihilator ideal, for each $a\in K$, $rl(RaR)\subseteq K$, hence $\bigvee_{a\in K}rl(RaR)\leq K$. On the other hand, $K\leq\sum_{a\in K}rl(RaR)\leq\bigvee_{a\in K}rl(RaR)$. Thus $K=\bigvee_{a\in K}rl(RaR)$. Since $K$ is compact, there are finitely elements $a_{1}, ..., a_{n}\in R$ such that $K=\bigvee_{i=1}^{n}rl(Ra_{i}R)$. This completes our proof.
\end{proof}
If $I$ is a right annihilator ideal, then $I=\bigvee_{a\in I}rl(RaR)$. This together with the above result imply that $rAnn(id(R))$ is an algebraic frame.
\begin{lem}\label{100}
Let $R$ be a semiprime ring.
Then $R$ is a reduced ring \ifif for each $a, b\in R$, $l(RaRbR)=l(RabR)$.
\end{lem}
\begin{proof}
$\Rightarrow$ Clearly $l(RaRbR)\subseteq l(RabR)$, since $RabR\subseteq RaRbR$. Now, assume $x\in l(RabR)$. Then $xab=0$, which implies $(bxa)^2=0$. By the reduced hypothesis, we obtain $bxa=0$ and so $b\in l(xaR)=r(xaR)$. Thus $xaRb=0$, which shows that $xaRbR=0$. This leads us to $x\in l(aRbR)=r(aRbR)$. Meaning that $aRbRx=0$ and thus $RaRbRx=0$. This implies $x\in r(RaRbR)=l(RaRbR)$.

$\Leftarrow$ Let $a\in R$ such that $a^2=0$. Then $l((RaR)^2)=l(Ra^2R)=R$, by hypothesis. Hence $(RaR)^2=0$. Since $R$ is semiprime, we conclude that $RaR=0$. Therefore, $a=0$.
\end{proof}
The goal is to show that $rAnn(id(R))$ is a coherent frame. We thus need to have
the meet of two compact elements. For that we apply Lemmas \ref{21} and \ref{100}.
\begin{thm}\label{23}
Let $R$ be a reduced ring. Then $rAnn(id(R))$ is a coherent  frame.
\end{thm}
\begin{proof}
We first note that  $R=rl(R1R)$ (i.e., the top element)  is compact in $rAnn(id(R))$, by Lemma \ref{21}. To see coherence, let $K_1, K_2\in\mathfrak{k}(rAnn(id(R)))$ with, say, \[K_1= rl(Ra_1R)\vee...\vee rl(Ra_kR) \quad\text{and}\quad K_2= rl(Rb_1R)\vee...\vee rl(Rb_nR).\] Then Lemmas \ref{20} and  \ref{100} imply \[K_1\wedge K_2=rl(Ra_1R)\cap rl(Rb_1R)\vee...\vee rl(Ra_1R)\cap rl(Rb_nR)\vee...\vee\]\[rl(Ra_kR)\cap rl(Rb_1R)\vee...\vee
rl(Ra_kR)\cap rl(Rb_nR)=\]\[rl(Ra_1Rb_1R)\vee...\vee rl(Ra_1Rb_nR)\vee...\vee rl(Ra_kRb_1R)\vee...\vee rl(Ra_kRb_nR)=\]\[rl(Ra_1b_1R)\vee...\vee rl(Ra_kb_nR).\] Now by using Lemma \ref{21}, $K_1\wedge K_2$ is compact. Therefore, $rAnn(id(R))$ is a
coherent frame.
\end{proof}
%%%%%%%%%%%%%%%%%%%%%%%%%%%%%%%%%%%%%%%%%%%%%%%%%%%%%%%%%%%%%
\begin{prop}
For each ideal $I$ of $R$, the smallest right annihilator ideal containing $I$ equals to the ideal
\[I_{A}=\bigvee_{a\in I}rl(RaR).\]
\end{prop}
\begin{proof}
It is evident $rl(I)$ is a right annihilator ideal containing $I$. Now, let $J$ be a right annihilator ideal containing $I$. Then $l(J)\subseteq l(I)$, and hence $rl(I)\subseteq rl(J))=J$. Thus, $rl(I)$ is the smallest right annihilator ideal containing $I$. We denote it by $I_{A}$ and  we claim that \[I_{A}=\bigvee_{a\in I}rl(RaR).\]To see it,  first note that $\bigvee_{a\in I}rl(RaR)$ is a right annihilator ideal, since $rAnn(id(R))$ is a complete lattice. Next, for each $a\in I$, $rl(RaR)\subseteq rl(I)=I_{A}$. Thus, \[I\subseteq\sum_{a\in I}rl(RaR)\subseteq\bigvee_{a\in I}rl(RaR)\subseteq I_{A}.\] Since $I_{A}$ is the smallest right annihilator ideal containing $I$, the proof is complete.
\end{proof}
%%%%%%%%%%%%%%%%%%%%%%%%%%%%%%%%%%%%%%%%%%%%%%%%%%%%%%%%%%
\begin{lem}\label{mary}
Let $I$ and $J$ be two ideals of $R$.
\begin{enumerate}
\item $I\subseteq J$ implies $I_A\subseteq J_A$.

\item If $R$ is a reduced ring, then $I\subseteq \sqrt{I}\subseteq I_A$. Hence $I_A=\sqrt{I}_A$.

\item If $R$ is a semiprime ring, then $(I\cap J)_A=I_A\cap J_A$.

\item $(I+J)_A=(I_A+J_A)_A$.
\end{enumerate}
\end{lem}
\begin{proof}
(1) Trivial.

(2) It is enough to show that $\sqrt{I}\subseteq I_A$. Consider $x\in\sqrt{I}$ and $r\in l(I)$. Then $x^n\in I$ for some $n\in\Bbb N$ and hence $rx^n=0$. By Lemmas \ref{20} and \ref{100} we have 
\[0=rl(Rrx^nR)=rl(RrRx^nR)=rl(RrR)\cap rl(Rx^nR)=\]\[rl(RrR)\cap rl(RxR)=rl(RrRxR)=rl(RrxR).\] This implies $l(RrxR)=R$ and hence $rx=0$, i.e., $x\in rl(I)=I_A$.

(3) By Lemma \ref{20},  we have \[I_A\cap J_A=rl(I)\cap rl(J)=rl(I\cap J)=(I\cap J)_A.\]

(4) We have,
\[(I+J)_A=rl(I+J)=rl(I)\cap rl(J)=\]\[rlr(l(I))\cap rlr(l(J))=\]\[rl(rl(I)+rl(J))=\]\[rl(I_A+ J_A)=(I_A+J_A)_A.\]
\end{proof}
%%%%%%%%%%%%%%%%%%%%%%%%%%%%%%%%%%%%%%%%%%%%%%%%%%%%%%%%%%%%%%%%%%%%%%%%%%%%%%%%%%%%%%%%%%%%%
\begin{cor}\label{helo}
The following statements hold.
\begin{enumerate}
\item For an ideal $I$ of a semiprime ring $R$, $I_A=R$ \ifif $I$ is an essential ideal.

\item A maximal ideal $M$ of a semiprime ring $R$ is a right annihilator ideal \ifif it is generated by an idempotent.
\end{enumerate}
\end{cor}
\begin{proof}
(1) The equality $rl(I)=I_A=R$ is equivalent to the $l(I)= 0$. By \cite[Corollary 14.2]{L1},  $l(I)=0$ \ifif $I$ is an essential ideal.

(2) It is well-known that every maximal ideal $M$ in a ring $R$ is either essential or generated by an idempotent. Now, if $M$ is a right annihilator ideal, then $M=r(l(M))=M_A$. By Part (1), this is equivalent to the $M$ being a non-essential ideal, i.e., $M$ must be generated by an idempotent.
\end{proof}
%%%%%%%%%%%%%%%%%%%%%%%%%%%%%%%%%%%%%%%%%%%%%%%%%%%%%%%%%%%%%%%%%%%%%%%%%%%%%%%%%%%%%%%%%%%%%%%%%%%%%%%%%%%%%%%%%%%%%%%%%%

In the next result, we provide a characterization of the largest (resp., smallest) annihilator ideal contained in (resp., consisting of) an ideal $I$ of an $SA$-ring $R$. We denote the largest right annihilator ideal contained in $I$ by $I^A$.
\begin{thm}\label{SA-theorem}
Let $R$ be a right Noetherian ring. The following statements
are equivalent.
\begin{enumerate}
\item For each ideal $I$ of $R$, the largest right annihilator
ideal contained in $I$ exists.

\item The ring $R$ is a right $SA$-ring.

\item For any two ideals $I$ and $J$ of $R$,
\[
(I+J)_A=I_A+J_A.
\]
\end{enumerate}
Furthermore, if these conditions hold, then for every ideal
$I$ of $R$,
\begin{equation}\label{eq:IA}
I_A=\sum_{a\in I}rl(RaR)
\tag{$\dagger$}
\end{equation}
and
\begin{equation}\label{eq:IA-lower}
I^A=
\sum_{\substack{a\in R\\rl(RaR)\subseteq I}}
rl(RaR).
\tag{$\ddagger$}
\end{equation}
\end{thm}

\begin{proof}
$(1)\Rightarrow(2)$
Let $I$ and $J$ be right annihilator ideals of $R$. By (1),
the largest right annihilator ideal $(I+J)^A$ contained in
$I+J$ exists. Since $I,J\subseteq(I+J)^A$, we obtain
\[
I+J\subseteq(I+J)^A\subseteq I+J.
\]
Hence $I+J=(I+J)^A$, so $I+J$ is a right annihilator ideal.
Therefore, $R$ is a right $SA$-ring.

$(2)\Rightarrow(3)$
Let $I$ and $J$ be ideals of $R$. Since $I_A$ and $J_A$
are right annihilator ideals, the right $SA$-property implies
that $I_A+J_A$ is a right annihilator ideal. Moreover,
$I+J\subseteq I_A+J_A$. By the minimality of $(I+J)_A$,
we have
\[
(I+J)_A\subseteq I_A+J_A.
\]
Conversely, $(I+J)_A$ is a right annihilator ideal containing
both $I$ and $J$. Hence, by the minimality of $I_A$ and
$J_A$,
\[
I_A+J_A\subseteq(I+J)_A.
\]
Consequently,
\[
(I+J)_A=I_A+J_A.
\]

$(3)\Rightarrow(2)$
Let $I$ and $J$ be right annihilator ideals of $R$. Then
$I_A=I$ and $J_A=J$. By (3),
\[
(I+J)_A=I_A+J_A=I+J.
\]
Thus $I+J$ is a right annihilator ideal, and $R$ is a
right $SA$-ring.

$(2)\Rightarrow(1)$
Assume that $R$ is a right $SA$-ring. Let $I$ be an ideal
of $R$, and consider
\[
I^A=\sum_{\substack{K\text{ a right annihilator ideal}\\
                     K\subseteq I}}K.
\]
The family in this sum is nonempty because the zero ideal
is a right annihilator ideal. Since $R$ is right Noetherian,
there are finitely many right annihilator ideals
$K_1,\ldots,K_m\subseteq I$ such that
\[
I^A=K_1+\cdots+K_m.
\]
By the right $SA$-property, the sum of finitely many right
annihilator ideals is a right annihilator ideal. Hence
$I^A$ is a right annihilator ideal contained in $I$.
By its definition, it contains every right annihilator ideal
contained in $I$. Therefore, $I^A$ is the largest such ideal.

We now prove the two asserted formulas.

For \eqref{eq:IA}, let
\[
S=\sum_{a\in I}rl(RaR).
\]
For each $a\in I$, we have $a\in rl(RaR)$; hence
\[
I\subseteq S.
\]
Since $R$ is right Noetherian, $S$ is a finite sum of
right annihilator ideals. The right $SA$-property therefore
implies that $S$ is a right annihilator ideal. Since $I_A$
is the smallest right annihilator ideal containing $I$,
we have $I_A\subseteq S$.

On the other hand, if $K$ is any right annihilator ideal
containing $I$, then for every $a\in I$,
\[
rl(RaR)\subseteq K.
\]
Indeed, writing $K=r(T)$, the inclusion $a\in K$ and the
fact that $K$ is a right ideal give $RaR\subseteq r(T)$,
so $T\subseteq\ell(RaR)$. Thus
$rl(RaR)\subseteq r(T)=K$.
It follows that $S\subseteq I_A$. Therefore,
\[
I_A=\sum_{a\in I}rl(RaR).
\]

Finally, consider
\[
T=\sum_{\substack{a\in R\\rl(RaR)\subseteq I}}
rl(RaR).
\]
Every summand is a right annihilator ideal contained in
$I$. Since $R$ is right Noetherian, $T$ is a finite sum
of such ideals; hence, by the right $SA$-property, $T$
is a right annihilator ideal contained in $I$. Thus
$T\subseteq I^A$.

Conversely, let $K$ be any right annihilator ideal contained
in $I$. For each $a\in K$, the preceding argument gives
\[
rl(RaR)\subseteq K\subseteq I.
\]
Consequently,
\[
K=\sum_{a\in K}rl(RaR)\subseteq T,
\]
where the equality follows from \eqref{eq:IA} applied to
$K$. Taking $K=I^A$, we obtain $I^A\subseteq T$.
Therefore,
\[
I^A=
\sum_{\substack{a\in R\\rl(RaR)\subseteq I}}
rl(RaR).
\]
This completes the proof.
\end{proof}

%%%%%%%%%%%%%%%%%%%%%%%%%%%%%%%%%%%%%%%%%%%%%%%%%%%%%%%%%%%%%%%%%%%%%%%%%%5
\begin{lem}\label{int}
Let $I_1, I_2,..., I_n$ be ideals of $R$ such that for each $1\leq i\neq j\leq n$, $I_i$ and $I_j$ be  co-prime. Then for each $1\leq j\leq n$, $I_j$ and $\bigcap_{i=1, i\neq j}^{n}I_i$ are co-prime.
\end{lem}
\begin{proof}
Fix $1\leq j\leq n$, we have $I_j+ I_i=R$ for each $1\leq i\leq n$ with $i\neq j$. Thus for each $1\leq i\leq n$ with $i\neq j$, there exist elements $a_{ij}\in I_j$ and $b_i\in I_i$ such that $1=a_{ij}+ b_i$. Multiplying these expressions for all $i\neq j$, we obtain  $1=\prod_{i=1, i\neq j}^n(a_{ij}+ b_i)$, which is of the form $x+y$, where $x\in I_j$ and $y\in\bigcap_{i=1, i\neq j}^{n}I_i$. So we are done.
\end{proof}
%%%%%%%%%%%%%%%%%%%%%%%%%%%%%
\begin{prop}\label{park}
Let $I_1, I_2,..., I_n$ be ideals of a semiprimme ring $R$ such that for each $1\leq i\neq j\leq n$, $I_i$ and $I_j$ be  co-prime. Then $\bigcap_{i=1}^{n}I_i$ is a right annihilator ideal \ifif each $I_j$ ($1\leq j\leq n$) is a right annihilator ideal.
\end{prop}
\begin{proof}
The necessity is obvious. Conversely, assume that $\bigcap_{i=1}^{n}I_i$ is a right annihilator ideal. Consider the ideal $I_j$ for some $1\leq j\leq n$ with $l(I_j)\subseteq l(a)$ and $a\in R$. By Lemma \ref{int}, $I_j$ and $\bigcap_{k=1, k\neq j}^{n}I_k$ are co-prime. Therefore $I_j+\bigcap_{k=1, k\neq j}^{n}I_k=R$. Hence $1=x+y$ for some $x\in I_j$ and $y\in \bigcap_{k=1, k\neq j}^{n}I_k$. Multiplying both sides by $a$, we get $a=ax+ ay$.  By Lemma \ref{into},  $l(I_jy)\subseteq l(ay)$. Now, since $I_jy\subseteq \bigcap_{k=1}^{n}I_k$ and $\bigcap_{k=1}^{n}I_k$ is a right annihilator ideal, we have $ay\in \bigcap_{k=1}^{n}I_k$. Thus $ay\in I_j$. Since $ax\in I_j$, we conclude that $a=ax+ay\in I_j$.
\end{proof}
%%%%%%%%%%%%%%%%%%%%%%%%%%%%%%%%%%%%%%%%%%%%%%%%%%%%%%%%%%%%%%%%%%%%%
Proposition \ref{park} does not hold for an infinite number of co-prime ideals. Consider the ring $\Bbb Z$ (i.e., the ring of integer numbers). Then, we have the intersection of all its maximal ideals is zero. But none of them is an annihilator ideal, by Corollary \ref{helo}.
\begin{prop}
The following statements hold.
\begin{enumerate}
\item Let $I$ be an  ideal of a semiprime ring $R$ and  $P$ be a prime ideal of  $R$ with
$I\cap P$  a right annihilator ideal. Then either $I$ or $P$ is a right annihilator ideal.

\item Let $P$ and $Q$ be prime ideals of a semiprime ring $R$ which do not belong
to a chain and $P\cap Q$ is a right annihilator ideal. Then both $P$ and $Q$ are
right annihilator ideals.

\item  Let $I$ be an ideal of a semiprime ring  $R$ and $M$ be a maximal ideal of $R$ such that
$I\not \subseteq M$ and $I\cap M$ is a right annihilator ideal. Then both ideals
$I$ and $M$ are  right annihilator ideals.
\end{enumerate}
\end{prop}
\begin{proof}
(1) If $I\subseteq P$, then $I=I\cap P$ is a right annihilator ideal. Now suppose that $I\not\subseteq P$ and $a\in R$ and $l(P)\subseteq l(a)$. Then, there exists $x\in I\setminus P$. By Lemma \ref{into},  $l(Px)\subseteq l(aRx)$. Since $I\cap P$ is a right annihilator ideal and $Px\subseteq I\cap P$, we have $aRx\subseteq I\cap P$, and hence $aRx\subseteq P$. As $x\not\in P$, we have $a\in P$. Hence, $P$ is a right annihilator ideal.

(2) Let $x\in Q\setminus P$ and $l(P)\subseteq l(a)$ for some $a\in R$. Then, $l(Px)\subseteq l(aRx)$, by Lemma \ref{into}. Since $P\cap Q$ is a right annihilator ideal and $Px\subseteq P\cap Q$, we have $aRx\in P\cap Q$. Hence, $aRx\in P$. This implies $a\in P$. Similarly, we can prove that $Q$ is a right annihilator ideal.

(3) Since $I\not\subseteq M$, $I$ and $M$ are co-prime and hence by Proposition \ref{park}, $I$ and $M$ are right annihilator ideals.
\end{proof}
%%%%%%%%%%%%%%%%%%%%%%%%%%%%%%%%%%%%%%%%%%%%%%%%%%%%%%%%%%%%%%
%\section*{Declarations}

%\subsection*{Ethical Approval}
%Not applicable.

%\subsection*{Competing interests}
%The author declares that there is no conflict of interest.

%\subsection*{Authors' contributions}
%This article was written by a single author.

%\subsection*{Funding}
%No funding was received for this research.

%\subsection*{Availability of data and materials}
%Not applicable.
\begin{center}
{\textbf{Acknowledgments}}
\end{center}
The author thanks the referee for their detailed report, which improved the initial manuscript, and Prof. Namdari for their assistance with the paper's presentation. 
 %------------------------------------------------------------------------------------%
\vskip 0.4 true cm

\bigskip
%\vspace{-0.2cm}
\begin{tabular}{l}
\vspace{-0.2cm}
{\scriptsize {\bf Author}} \\
\vspace{-0.2cm}
 {\scriptsize Department of Mathematics,}\\
 \vspace{-0.2cm}
 {\scriptsize Faculty of  Sciences,}\\ 
 \vspace{-0.2cm}
 {\scriptsize Yasouj University}\\ 
 \vspace{-0.2cm}
 {\scriptsize Yasouj, Iran}\\
 \vspace{-0.2cm}
{\scriptsize Email: ataherifar@yu.ac.ir, ataherifar54@gmail.com}
\end{tabular}
\Small{
\vspace{0.2cm}
\hrule
\vspace{0.2cm}
%\noindent \includegraphics[width=1.2cm]{copy_right.jpg} 
\textcopyright 2024 Shahid Chamran University of Ahvaz, Ahvaz, Iran. This article is an open-access article distributed under the terms and conditions of the Creative Commons Attribution-NonComertial 4.0 International (CC BY-NC 4.0 license)
(http://creativecommons.org/licenses/by-nc/4.0/).}
%\end{large}
\end{document}